\documentclass[11pt]{article}
\usepackage{a4wide}
\usepackage{latexsym, amsmath, amsfonts, amsthm, amssymb}
\usepackage{ulem}

\def\id{\textrm{Id}}
\def\RR{\mathbb{R}}
\def\vphi{\varphi}
\def\eps{\varepsilon}
\newcommand{\D}{\ensuremath{\mathbb{D}}} 
\newcommand{\R}{\ensuremath{\mathbb{R}}} 
\newcommand{\C}{\ensuremath{\mathbb{C}}} 

\newcommand{\longto}{\longrightarrow}

\newtheorem{theo}{Theorem}
\newtheorem{prop}[theo]{Proposition}

\newtheorem{notation}[theo]{Notation}
\newtheorem{fact}[theo]{Fact}
\newtheorem{conj}[theo]{Conjecture}

\newtheorem{defi}{Definition}

\def\D{\displaystyle}

\def\longto{\longrightarrow}
\def\D{\displaystyle}

\newcommand{\be}{\begin{equation}}
\newcommand{\ee}{\end{equation}}

\def\benu{\begin{enumerate}}
\def\eenu{\end{enumerate}}

\def\ov#1{\overline{#1}}
\def\Hess{{\rm Hess}}
\def\L{\mathcal L}

\def\myffrac#1#2 in #3{\raise 2.6pt\hbox{$#3 #1$}\mkern-1.5mu\raise 0.8pt\hbox{$
#3/$}\mkern-1.1mu\lower 1.5pt\hbox{$#3 #2$}}
\newcommand{\ffrac}[2]{\mathchoice%
    {\myffrac{#1}{#2} in \scriptstyle}
    {\myffrac{#1}{#2} in \scriptstyle}
    {\myffrac{#1}{#2} in \scriptscriptstyle}
    {\myffrac{#1}{#2} in \scriptscriptstyle}
}

\begin{document}

\title{Interpolations, convexity and geometric inequalities}

\author{D. Cordero-Erausquin\textsuperscript{1} and B. Klartag\textsuperscript{2}}

\date{}

\footnotetext[1]{Institut de Math\'ematiques de Jussieu,
Universit\'e Pierre et Marie Curie (Paris 6), 4 place Jussieu, 75252
Paris, France. Email: cordero@math.jussieu.fr.}

\footnotetext[2]{School of Mathematical Sciences, Tel-Aviv
University, Tel Aviv 69978, Israel. Supported in part by the Israel
Science Foundation and by a Marie Curie Reintegration Grant from the
Commission of the European Communities. Email: klartagb@tau.ac.il}

\maketitle

\begin{abstract}
We survey some interplays between spectral estimates of
H\"ormander-type,   degenerate Monge-Amp\`ere equations and
geometric inequalities related to log-concavity such as
Brunn-Minkowski, Santal\'o or Busemann inequalities.
\end{abstract}

\section{Introduction}

The Brunn-Minkowski inequality has an $L^2$ interpretation, an
observation that can be traced back to the proof provided by
Hilbert. More recently, it has been noted that the Brunn-Minkowski
inequality for convex bodies is related, in its local form, to
spectral inequalities. In fact, the Pr\'ekopa theorem, which is the
function form of the Brunn-Minkowski inequality for convex sets, is
{\it equivalent} to spectral inequalities of Brascam-Lieb type. The
local derivation of Pr\'ekopa's theorem from spectral $L^2$
inequalities was described in the more general complex setting
in~\cite{CE05} and then extended further
in~\cite{Berndtsson:2008p480, Berndtsson:2009p2332}.

Let $K_0,K_1\subset \R^n$ be two convex bodies (i.e., compact convex
sets with non-empty interior) and denote, for $t\in [0,1]$,
\begin{equation} K(t) := (1-t)K_0 +t K_1= \{z\in \R^n\; ; \ \exists
(a,b)\in K_0\times K_1 , \; z=(1-t)a + tb\}. \label{eq_355}
\end{equation} The Brunn-Minkowski inequality is central in the theory of
convex bodies. Denoting the Lebesgue measure by $|\cdot|$, it states
that
$$|K(t)|\ge |K_0|^{1-t} \, |K_1|^t,$$
with equality if and only if $K_0=K_1 + x_0$ for $x_0\in \R^n$.
Introducing the convex body
$$K:= \bigcup_{t\in [0,1]} \{t\}\times K(t) \subset \R^{n+1},$$
then $K(t)$ is the section over $t$, and the Brunn-Minkowski
inequality expresses the log-concavity of the marginal measure.
Namely, it shows that the function
$$\alpha(t):=-\log|K(t)|$$
is convex.  The Brunn-Minkowski inequality for convex bodies admits
the following  useful functional form, which states that marginals
of log-concave functions are log-concave.

\begin{theo}[Pr\'ekopa]\label{prekopa}
Let $F:\R^{n+1} \to \R\cup\{+\infty\}$ be convex with $\int \exp(-F)
< \infty$ and define $\alpha:\R \longto \R\cup\{+\infty\}$ by
$$e^{-\alpha(t)}= \int_{\R^n} e^{-F(t,x)}\, dx .$$
Then $\alpha$ is convex.
\end{theo}
The Brunn-Minkowski inequality then follows by considering, for a
given convex set $K\subset \R^{n+1}=\R\times \R^n$, the convex
function $F$ defined by
\begin{equation}\label{indicatrix}
e^{-F(t,x)}= \mathbf 1_{K}(t,x) = \mathbf 1_{K(t)}(x) .
\end{equation}

The standard proofs of Brunn-Minkowski rely on parameterization or
mass transport techniques between $K_0$ and $K_1$, with the
parameter $t\in [0,1]$ being fixed. A natural question is whether
one can provide a direct local approach by proving $\alpha''(t) \ge
0$? The answer is affirmative and this was shown recently by Ball,
Barthe and Naor \cite{Ball:2003p54}. As mentioned earlier, this
local approach was put forward in an $L^2$ framework, for analogous
complex versions, in Cordero-Erausquin~\cite{CE05} and in subsequent
far-reaching works by Berndtsson~\cite{Berndtsson:2008p480,
Berndtsson:2009p2332}.

Another essential concept in the theory of convex bodies is duality.
This requires us to fix a center and a scalar product. Let $x \cdot
y$ stand for the standard scalar product of $x, y \in \R^n$. We
write $|x|^2 = x \cdot x$ and $B_2^n=\{x\in \R^n\; ;\ x\cdot x\le
1\}$, the associated unit ball. Recall that $K \subset \R^n$ is a
centrally-symmetric convex body if and only if $K$ is the unit ball
for some norm $\|\cdot\|$ on $\R^n$, a relation denoted by
$K=B_{\|\cdot\|}:=\{x\in \R^n\; ; \ \|x\|\le 1\}$. The polar of $K$
is defined as the unit ball of the dual norm $\|\cdot\|_\ast$,
$$K^\circ = B_{\|\cdot\|_\ast}= \{y\in \R^n\; ; \ x\cdot y \le 1, \; \forall x \in K\}.$$
We have the following beautiful result:

\begin{theo}[Blaschke-Santal\'o inequality]
For every centrally-symmetric convex body $K\subset \R^n$, we have
$$|K|\, |K^\circ| \le |B_2^n|^2$$
with equality holding true if and only if $K$ is an ellipsoid (i.e.
a linear image of $B_2^n$).
\end{theo}

The corresponding functional form reads as follows
(see~\cite{ArtsteinAvidan:2004p118, Ball_phd}): for an even function
$f:\R^n\to \R$ with $0 < \int e^{-f} < \infty$, if $\mathcal Lf$
denotes its Legendre transform, then
\begin{equation}\label{functsantalo}
\int e^{-f}\int e^{-\mathcal L f} \le \Big(\int e^{-|x|^2/2}\, dx\Big)^2 = (2\pi)^n.
\end{equation}
Note that the Brunn-Minkowski inequality entails
$$ \sqrt{ |K|\, |K^\circ| } \le \left|\frac{K+K^\circ}2\right|.$$
However, in general we have $\frac{K+K^\circ}2 \varsupsetneq
B_2^n$. For instance, take $K=T(B_2^n)$, where $T\neq
\textrm{Id}_{\R^n}$ is a positive-definite symmetric operator. Then
$K^\circ = T^{-1}(B_2^n)$. Observe that $\frac{K+K^\circ}2
\supset \frac{T+T^{-1}}2(B_2^n)$ and
$$ \frac{T+T^{-1}}2 > \sqrt{T \, T^{-1}} = \textrm{Id}_{\R^n} $$
in the sense of symmetric matrices. This suggest that instead of
taking convex combinations, as in  the Brunn-Minkowski theory, we
would like to consider geometric means of convex bodies. It turns
out that this is exactly what complex interpolation does, and it is
a challenging question to understand real analogues of this
procedure.

In this note we will consider several ways of going from $K_0$
to $K_1$, or equivalently from a norm $\|\cdot \|_0$ to another norm
$\|\cdot \|_1$. There are many ways to recover the volume of $K$
from the associated norm $\| \cdot \|$. Let $p>0$ and $n\ge 1$.
There exists an explicit constant $c_{n,p}>0$ such that for every
centrally-symmetric convex body $K\subset \R^n$, with associated
norm $\|\cdot\|_K$, we have
\begin{equation}\label{volume}
\int_{\R^n} e^{-\|x\|_K^p / p} \, dx = c_{n,p}\, |K|.
\end{equation}
Note that the procedure~\eqref{indicatrix} corresponds to the case $p\to+\infty$.

We  aim to find ways of interpolating between norms in order to
recover, among other things,  the Brunn-Minkowski and the Santal\'o
inequalities.

Let us next put forward some notation as well as a formula that we shall use
throughout the  paper.

\begin{notation}\label{notation}
For a  function $F:\R^n\to \R$ such that $\int e^{-F(x)}\,
dx<+\infty$, we denote by $\mu_{F}$ the {\it probability} measure on
$\R^n$ given  by
$$d\mu_{F}(x):= \frac{\D e^{-F(x)}}{\int e^{-F}} \, dx.$$

For a function of $n+1$ variables  $F:I\times \R^n \to \R$, where
$I$ is an interval of $\R$, we denote, for a fixed $t\in I$,   $F_t:=F(t,\cdot):\R^n\to \R$
and then by $ \mu_{F_t}$ the corresponding probability measure on
$\R^n$.  We also set
$$\alpha(t) =  - \log\int_{\R^n} e^{-F_t(x)} dx. $$
\end{notation}

The variance with respect to a {\it probability} measure $\mu$ of a
function $u\in L^2(\mu)$ -- where, depending on the context, we
consider either real-valued or complex-valued functions -- is defined as the
$L^2$ norm of the projection of $u$ onto the space of functions
orthogonal to constant functions, i.e.
$${\rm Var}_{\mu}(u):= \int \left| u - \mbox{$\int$} u\, d\mu \right|^2 \, d\mu= \int |u|^2\, d\mu - \Big| \int u \, d\mu \Big|^2.$$
A straightforward  computation yields:

\begin{fact}
With  Notation~\ref{notation}, we have for every $t\in I$,
\begin{eqnarray}
\alpha''(t) & = & \int_{\R^n} \partial^2_{tt} F \, d\mu_{F_t}(x)- \left[\int_{\R^n} \big(\partial_t F(t,x)\big)^2\, d\mu_{F_t}(x) - \left( \int_{\R^n} \partial_t F(t,x)\, d\mu_{F_t}(x)\right)^2 \right] \nonumber \\
& = & \int_{\R^n} \partial^2_{tt} F \, d\mu_{F_t} - {\rm
Var}_{\mu_{F_t}} \big( \partial_t F\big), \label{formula}
\end{eqnarray}
assuming that $F$ is sufficiently regular to allow for the
differentiations under the integral sign.
\end{fact}

Our goal is to understand for which families of functions $F$ the
function $\alpha$ is convex, by looking at $\alpha''$. Actually, we
will first discuss the complex case, where convexity is replaced by
plurisubharmonicity.  We will recover the fact that families given
by complex interpolation, or equivalently by degenerate Monge-Ampère
equations,  lead to subharmonic functions $\alpha$. Then we will try
to see, at a very heuristic level, what can be said in the real
case. A final section proposes a local $L^2$ approach, to the
Busemann inequality, similar to that used in the preceding sections.

\medskip {\it{Acknowledgement.}}
We thank Yanir Rubinstein and Bo Berndtsson for interesting, related discussions.


\section{The complex case}

Let $K_0$ and $K_1$ be two unit balls of $\C^n$ associated with the
(complex vector space) norms $\|\cdot \|_0$ and $\|\cdot\|_1$. Note
that here we are working with the class of convex bodies $K$ of
$\R^{2n}$ that are {\it circled}, meaning that $e^{i\theta} K = K$
for every $\theta \in \R$. We think of a  normed space as a triplet
consisting of a vector space, a norm and its unit ball. Consider the
complex normed spaces $X_0=(\C^n, \|\cdot \|_0, K_0)$ and
$X_1=(\C^n, \|\cdot\|_1, K_1)$ and write
$$X_z = (\C^n , \|\cdot\|_z , K_z)$$
for the complex Calder\'on interpolated space at
$$z\in C:= \{w \in \C \; ; \Re(w) \in [0,1]\} $$
where $\Re(w)$ is the real part of $w \in \C$. Recall that $X_z =
X_{\Re(z)}$ and therefore  $K_z = K_{t}$ with $t=\Re(z)\in [0,1]$.
We have:
\begin{theo}[\cite{CE02}]
The function $t\to|K_t|$ is log-concave on $[0,1]$ and so
\begin{equation}
\label{santalo_interpol}
|K_0|^{1-t}\, |K_1|^t \le |K_t| .
\end{equation} \label{th1037}
\end{theo}
In the case of complex unit balls, this result improves upon the
Brunn-Minkowski inequality since it can be  verified, by using the
Poisson kernel on $[0,1]\times \C^n$ and the definition of the
interpolated norm,  that
$$K_t \subset (1-t) K_0 + t K_1=K(t).$$
In this setting,  it also gives the Santal\'o inequality. Indeed,
for a given complex unit ball $K\subset \C^n$, let $X_0$ be the
associated complex normed space, and let $X_1$ be the dual conjugate
space which has $K^\circ \subset \C^n$ as its unit ball. Then it is
well known that
\begin{equation} X_{1/2} = \ell_2^n (\C)= \ell_2^{2n}(\R)
\label{eq_1035}
\end{equation}
and therefore we obtain
$$\sqrt{|K|\, |K^\circ|} \le |B_2^{2n}|.$$
(Let us mention here that the conjugation bar in the statements of~\cite{CE02} is superfluous according to standard definitions).

\medskip
 In order to have a
better grasp on complex interpolation, let us write an explicit
formula in the specific case of Reinhardt domains. A subset $K
\subset \C^n$ is {\it Reinhardt} if for any $z = (z_1,\ldots,z_n)
\in \C^n$,
$$ (z_1,\ldots,z_n) \in K \quad \Leftrightarrow \quad
(|z_1|,\ldots,|z_n|) \in K. $$ Note that a Reinhardt convex set is
necessarily circled. In the case where $X_0=(\C^n, \|\cdot \|_0,
K_0)$ and $X_1=(\C^n, \|\cdot\|_1, K_1)$ are such that $K_0$ and
$K_1$ are Reinhardt, the interpolated space $X_z = (\C^n ,
\|\cdot\|_z , K_z)$ satisfies $$ K_z = \left \{ z \in \C^n \, ; \,
\exists (a,b) \in K_0 \times K_1, \ |z_j| = |a_j|^{1-t} |b_j|^{t } \
\text{for} \ j=1,\ldots,n \right \}
$$ with $t = \Re(z)$. The case of Reinhardt unit balls is
particularly simple and easy to analyze, but it has its limitations.
Still, the idea is that in general, $K_t$ should be understood as a
``geometric mean'' of the bodies $K_0$ and $K_1$, whereas the
Minkowski sum (\ref{eq_355}) reminds us of an arithmetic mean.

\medskip Theorem \ref{th1037} was proved using the complex version of the Pr\'ekopa theorem
obtained by Berndtsson~\cite{Berndtsson:1998p255}, which was derived
in~\cite{CE05} using a local computation and $L^2$ spectral
inequalities of H\"ordmander type. Here, we would like to provide a
different direct proof, by combining the results of Rochberg and
H\"ormander's {\it a priori} $L^2$-estimates. Let $\|\cdot\|_z$ be a
family of interpolated norms on $\C^n$ and $K_z= B_{\|\cdot\|_z}$.
We assume for simplicity that these norms are smooth and strictly
convex, so that we will not have to worry about justification of the
differentiations under the integral signs. In fact, by approximation
we can assume that $1/R \leq \Hess \|\cdot \|^2_k \leq R$ (for some
large constant $R > 1$) for $k=1,2$, and these bounds remain valid
for the interpolated norms. Introduce the function $F:C\times \C^n
\to \R$,
$$F(z,w) :=\frac12 \|w\|_z^2.$$
Denote the Lebesgue measure on $\C^n\simeq \R^{2n}$ by $\lambda$,
and introduce, in view of~\eqref{volume},
$$\alpha(z) = - \log \int_{\C^n}e^{-F(z,w) } \, d\lambda(w) = - \log|K_z| - \log(c_{2n,2}) $$
for $z\in C$. Our goal is to prove that $t\to \alpha(t)$ is convex
on $[0,1]$. Since $\alpha(z)= \alpha(\Re(z))$, this is equivalent to
proving that $\alpha$ is subharmonic on the strip $C$. The following
analogue of~\eqref{formula} is also straightforward:
$$\frac14 \Delta\alpha(z) = \partial^2_{z \ov{z}}  \alpha (z) = \int_{\C^n} \partial^2_{z \ov{z}}  F \, d\mu_{F_z} -\int_{\C^n} \left| \partial_z F (w) - \mbox{$\int$}\partial_z F \, d\mu_{F_z} \right|^2 \, d\mu_{F_z}(w),$$
where $\mu_{F_z}$ is the probability measure on $\C^n$ given by $\D
d\mu_{F_z}(w) = \frac{e^{-F(z,w)}}{\int
e^{-F(z,\zeta)}d\lambda(\zeta)} d\lambda(w)$.

It was explained by Rochberg \cite{Rochberg84} that complex
interpolation is characterized by the following differential
equation:
\begin{equation}\label{rochberg}
\partial^2_{z\ov{z}} F = \sum_{j,k=1}^n F^{j\ov{k}} (z,w) \partial_{\overline{w_j}}(\partial_z F) \overline{ \partial_{\overline{w_k}} (\partial_z F)}
\end{equation}
where $ (F^{j\ov{k}})_{j,k\le n}$ is the inverse of the complex
Hessian in the $w$-variables of $F(z,w)$, that is
$$\left(F^{j\ov{k}}\right)_{j,k\le n} = \left(\Hess^{\C}_w F\right)^{-1} := \left[\left(\partial^2_{w_j\ov{w_k}} F \right)_{j,k\le n}\right]^{-1}.$$
Actually, the function $F$ is plurisubharmonic on $C\times \C^n
\subset \C^{n+1}$ and~\eqref{rochberg} expresses the fact that it is
a solution of the degenerate Monge-Amp\`ere equation
$$\det\Big(\Hess^{\C}_{z,w} F \Big)=0 $$
where $\Hess^{\C}_{z,w} F$ is the full complex Hessian of $F(z,w)$,
an $(n+1) \times (n+1)$ matrix.

As a consequence of the previous discussion, we have that, for a fixed $z\in C$ and setting $u:=\partial_z F(z, \cdot) :
\C^n \to \C$,
\begin{equation}\label{complexformula}
\Delta\alpha(z)/4=\int_{\C^n}  \sum_{j,k=1}^n F^{j\ov{k}}
\partial_{\overline{w_j}}u \, \overline{ \partial_{\overline{w_k}}u} \, d\mu_{F_z} - \int
\left| u - \mbox{$\int$}u \, d\mu_{F_z} \right|^2 \, d\mu_{F_z}.
\end{equation}
Of course, it is now irresistible to appeal to H\"ormander's
{\it a priori} estimate (see e.g.~\cite{H1}). It states that if $F:\C^n\to \R$ is a (strictly)
plurisubharmonic function and if $u$ is a (smooth enough) function,
then
\begin{equation}\label{hormander}
 \int_{\C^n} \left| u -P_H u \right|^2\, d\mu_{F} \le  \int_{\C^n}  \sum_{j,k=1}^n F^{j\ov{k}} \partial_{\overline{w_j}}u \, \overline{ \partial_{\overline{w_k}}u} \, d\mu_{F}
 \end{equation}
where $\D d\mu_{F}(w) = \frac{e^{-F(w)}}{\int e^{-F}d\lambda}
d\lambda(w)$ and $P_H:L^2(\mu_F) \to L^2(\mu_F)$ is the orthogonal
projection onto the closed space $H=\{h\in L^2(\mu_F)\; ;\
\ov{\partial} h=0\}$ of holomorphic functions. Actually, this {\it a
priori} estimate on $\C^n$  is rather easy to prove by duality and
integration by parts. We now apply this result to $F=F(z, \cdot)$,
$\mu_F=\mu_{F_z}$ and $u=\partial_z F$. Note that $F$ (and thus
$\mu_F)$ and $u$ are invariant under the action of $S^1$: $F(z,
e^{i\theta}w) = F(z, w)$ and the same is true for $\partial_z F$.
This implies that the function $P_H u$ has the same invariance, but
since it is a holomorphic function on $\C^n$, it has to be constant.
Therefore $P_H u = \int u d\mu_{F_z}$ and we indeed obtain  that
$\Delta \alpha(z) \ge 0$ by combining~\eqref{complexformula}
and~\eqref{hormander}, as desired.

Here, we reproved~\eqref{santalo_interpol} without using
explicitly~\cite{Berndtsson:1998p255}, but rather by combining the
local computations of~\cite{CE05} and the degenerate Monge-Amp\`ere
equation satisfied by the complex interpolation. In fact, this
computation also appears, in a much more general and deep form, in
recent works by Berndtsson~\cite{Berndtsson:2008p480,
Berndtsson:2009p2332}. The reason is that complex interpolation
corresponds to a geodesic in the space of metrics, and therefore
enters Berndtsson's abstract theorems. Also, it can be noticed that
complex interpolation corresponds to an extremal construction (for
given boundary data), in the sense that it can be viewed as a
plurisubharmonic hull. Equivalently, plurisubharmonic functions may
be viewed as sub-solutions of degenerate Monge-Amp\`ere equations.

Following our presentation, it is very tempting to develop an
analogous presentation for convex bodies in $\R^n$. However, the
real case is more complex, as we shall now see.


\section{Real interpolations}

The concept of interpolation and the basic properties we present
here are due to Semmes~\cite{Semmes88}, building on previous work by
Rochberg \cite{Rochberg84}. Semmes indeed raised the question of
whether such interpolations (which are not interpolations in the
operator sense) could be used to prove inequalities, by showing that
certain functionals are convex along the interpolation. Our main
contribution here is to explain that this is indeed the case, by
connecting this interpolation with some well-known spectral
inequalities. However, some discussions will remain at a heuristic
level, as it is not the purpose of this note to discuss existence,
unicity and regularity of solutions to the partial differential
equations we refer to.

\begin{defi}[Rochberg-Semmes interpolation~\cite{Semmes88}]
Let $I$ be an interval of $\R$ and $p\in [1, +\infty]$. We say that
a smooth function $F:I \times \R^n\to \R$ is a family of
$p$-interpolation if for any $t \in I$, the function $F(t,\cdot)$ is
(strongly) convex on $\R^n$ and for $(t,x)\in I\times \R^n$
\begin{equation}\label{kinterpol}
\partial^2_{tt} F = \frac1{p} \big(\Hess_x F\big)^{-1}\nabla\partial_t F \cdot \nabla \partial_t F .
\end{equation}
Accordingly, when $\partial^2_{tt} F \ge \frac1p \big(\Hess_x
F\big)^{-1}\nabla\partial_t F \cdot \nabla \partial_t F $, we say
that $F$ is a sub-family of $p$-interpolation. \label{def_1100}
\end{defi}

In Definition \ref{def_1100}, we denote by $\nabla F$ the gradient
of $F(t,x)$ in the $x$ variables, and a function is strongly convex
when $\Hess_x F>0$. By standard linear algebra we have the following
equivalent formulation in terms of the degenerate Monge-Amp\`ere
equation:

\begin{prop}[Interpolation and degenerate Monge-Amp\`ere equation]
Let $F:I \times \R^n\to \R$ be a smooth function such that
$F(t,\cdot)$ is (strongly) convex on $\R^n$ and introduce, for
$(t,x)\in I\times \R^n$, the $(n+1)\times (n+1)$ matrix
\begin{equation}\label{def:H}
H=H_p F (t,x):=
\begin{pmatrix}
\partial^2_{tt} F &  (\nabla_x \partial_t F)^\ast \\
 &  \\
\nabla_x \partial_t F & \D p \mathnormal{\Hess}_x F\\
\end{pmatrix}.
\end{equation}
Then, $F$ is a family (resp. a sub-family)  of $p$-interpolation if and only if $\D \det H = 0$ (resp. $\det H \ge 0$) on $I\times \R^n$.
\end{prop}
In particular, $1$-interpolation corresponds exactly to the
degenerate Monge-Amp\`ere equation on $I\times \R^n$. In fact, we
see $p$-interpolation as a (Dirichlet) boundary value problem.

\begin{defi}
Let $F_0$ and $F_1$ be two smooth convex functions on $\R^n$. We say
that $\{F_t:\R^n\to \R\}_{\in [0,1]}$ is a $p$-interpolated family
associated with $\{F_0, F_1\}$ if $F(t,x)=F_t(x)$ is a family of
$p$-interpolation on $[0,1]\times \R^n$ with boundary value
$F(0,\cdot) = F_0$ and $F(1,\cdot) =F_1$.
\end{defi}

\medskip As we said above, we will not discuss in this exposition questions related to
existence, uniqueness and regularity of solutions to this Dirichlet
problem (except for the easy case $p=1$, explained below). However,
it is reasonable to expect that generalized solutions, which are
sufficient for our purposes, can be constructed by using Perron
processes, as mentioned by Semmes~\cite{Semmes88}.

Using Notation~\ref{notation}, given  a family or a sub-family of
$p$-interpolation $F$,
we aim to understand the convexity of the function on $I$,
\begin{equation}\label{defalpha}
\alpha(t)= - \log\int_{\R^n} e^{-F(t,x)}\, dx.
\end{equation}
In view of~\eqref{formula}, we see that for every fixed $t\in I$ we
have the implication
\begin{equation}\label{convexity}
\textrm{Var}_{\mu_{F_t}} (\partial_t F) \le \frac1p \int_{\R^n}
\big(\Hess_x F\big)^{-1}\nabla\partial_t F \cdot \nabla \partial_t F
\, d\mu_{F_t}  \quad \Longrightarrow \quad \alpha''(t) \ge 0,
\end{equation}
under some mild regularity assumptions. The left-hand side is of
course reminiscent of the real version of H\"ormander's
estimate~\eqref{hormander}, which is known as the Brascamp-Lieb from
\cite{BL:1976}. Recall that this inequality states that if
$F:\R^n\to \R$ is a (strongly) convex function and if $u\in
L^2(\mu_F)$ is a locally Lipschitz  function, then
\begin{equation}\label{BL}
\textrm{Var}_{\mu_F}(u)\le  \int_{\R^n}  \big(\Hess_x F\big)^{-1} \nabla u \cdot \nabla u\, d\mu_{F} ,
 \end{equation}
with our  notation $\D d\mu_{F}(x) = \frac{e^{-F(x)}}{\int e^{-F}} \, dx$.
Again, this inequality can easily be proven along the lines of
H\"ormander's approach (see below).

Applying the Brascamp-Lieb inequality~\eqref{BL} to $F=F(t,\cdot)$
and $u=\partial_t F$ when $F$ is a $1$-interpolation sub-family, we
obtain, in view of~\eqref{convexity}, the following statement:
\begin{prop}
If $F$ is a sub-family of $1$-interpolation, then $\alpha$ is convex.
\end{prop}
The first comment is that we have not proved anything new! Indeed,
it is directly verified below that for any $C^2$-smooth function
$F$,
\begin{equation} F \textrm{ is a sub-family of $1$-interpolation} \quad \Longleftrightarrow \quad F \textrm{ is convex on } I\times
\R^n. \label{eq_1042}
\end{equation}
Therefore, we have reproduced Pr\'ekopa's Theorem~\ref{prekopa}. In
order to demonstrate \eqref{eq_1042}, observe that the positive
semi-definiteness of the matrix $H_1 F(t,x)$ amounts to the
inequality
$$ (\mathnormal{\Hess}_x F) y \cdot y + 2 \nabla_x (\partial_t F)
\cdot y + \partial^2_{tt} F  \geq 0 \quad \quad \text{for all} \ \ y
\in \R^n, $$ or equivalently,
$$ \partial^2_{tt} F \geq \sup_{y \in \R^n}  \left[ 2 \nabla_x (\partial_t
F) \cdot y - (\mathnormal{\Hess}_x) F y \cdot y \right] =
\big(\Hess_x F\big)^{-1}\nabla_x \partial_t F \cdot \nabla_x
\partial_t F, $$ as $\Hess_x F$ is positive definite. Let us note
that if $F_0$ and $F_1$ are given, then the associated family of
$1$-interpolation -- equivalently, the unique solution to the
degenerate Monge-Amp\`{e}re equation on $[0,1]\times \R^n$ with
$F(t, x)$ convex in $x$ -- is
\begin{equation}\label{infconv}
F(t,w)= \inf_{w=(1-t) x + t y} \big\{ \, (1-t) F_0(x) + tF_1(y)\, \big\} .
\end{equation}
Every sub-family of $1$-interpolation is {\it above} this $F$, and
thus the statement of Pr\'ekopa's Theorem reduces to
$1$-interpolation families (an argument that is standard in the
study of functional Brunn-Minkowski inequalities). One way to
recover the Brunn-Minkowski inequality directly from this family $F$
of $1$-interpolation, is to take, as in the derivation from
Pr\'ekopa's theorem, something like $F_0 (x) =\|x\|_{K_0}^q/q$, $
F_1(y):=\|y\|_{K_1}^q/q$ and let $q\to +\infty$.

We have just shown that Pr\'ekopa's theorem reduces, locally, to the
Brascamp-Lieb inequality. This is parallel to the complex setting,
i.e to the local $L^2$-proof  of the complex Pr\'ekopa theorem of
Berndtsson given in~\cite{CE05} and extended
in~\cite{Berndtsson:2008p480, Berndtsson:2009p2332}. The converse
procedure was known, starting from the work of  Brascamp and Lieb;
more explicitely,   Bobkov and Ledoux \cite{Bobkov:2000p81} noted
that the Pr\'ekopa-Leindler inequality (an extension of Pr\'ekopa's
result to the case fibers are not convex) indeed implies the
Brascamp-Lieb inequality. We also emphasize Colesanti's work
\cite{Colesanti:2008p2467},  where, starting from the
Brunn-Minkowski inequality, spectral inequalities of Brascamp-Lieb
type on the boundary $\partial K$ of a  convex body $K\subset \R^n$
are obtained. This can also be recovered by applying the
Brascamp-Lieb inequality to homogeneous functions. The conclusion is
that all of these results are the global/local versions of the same
phenomena. At the local level, we have reduced the problem to the
inequality~\eqref{BL} which expresses a spectral bound in
$L^2(\mu_F)$ for the  elliptic operator associated with the
Dirichlet form on the right-hand side of~\eqref{BL}.

For completeness, we would like to briefly recall here H\"ormander's
original approach to~\eqref{BL}. Consider the Laplace-type operator on $L^2(\mu_F)$,
$$L:=\Delta - \nabla F \cdot \nabla ,$$
that we define, say, on  $C^2$-smooth compactly supported functions.
First, recall the integration by parts formulae, $\int uL\varphi \,
d\mu_F = -\int \nabla u\cdot \nabla \varphi \, d\mu_F$ and
\begin{equation}\label{ipp}
\int_{\R^n} (L\varphi)^2 \, d\mu_F = \int_{\R^n} (\Hess_x F) \nabla
\varphi \cdot \nabla \varphi \, d\mu_F  + \int_{\R^n} \|\Hess_x
\varphi \|_2^2 \, d\mu_F,
\end{equation}
where $ \|\Hess \varphi \|_2^2  = \sum_{i,j\le n} (\partial_{i,j}^2
\varphi )^2.$
Let $u$ be a locally-Lipschitz function on $\R^n$. We use the
(rather weak) standard observation that the image by $L$ of the
$C^2$-smooth compactly supported functions is dense in the space of
$L^2(\mu_F)$ functions orthogonal to constants (see e.g.
\cite{CFM}). For $\eps > 0$ let $\vphi$ be a $C^2$-smooth,
compactly-supported function such that $L \vphi - (u - \int u d
\mu_F)$ has $L^2(\mu_F)$-norm smaller than $\eps$. Then, by
integration by parts and using~\eqref{ipp} we get
\begin{eqnarray*}
{\rm Var}_{\mu_F}(u) &=& 2 \int  \big(u- \mbox{$\int$} u\, d\mu_F\big) L\varphi \, d\mu_F - \int (L\varphi)^2 \, d\mu_F + \int \left( L \vphi -  \big(u- \mbox{$\int$} u\, d\mu_F\big)\right)^2 d \mu_F \\
& \le & -2\int \nabla u \cdot \nabla \varphi \, d\mu_F -  \int (\Hess_x F) \nabla \varphi \cdot \nabla \varphi \, d\mu_F  - \int \|\Hess_x \varphi \|_2^2 \, d\mu_F + \eps^2 \\
&\le &  -2\int \nabla u \cdot \nabla \varphi -  \int (\Hess_x F) \nabla \varphi \cdot \nabla \varphi \, d\mu_F + \eps^2 \\
& \le &  \int  \big(\Hess_x F\big)^{-1} \nabla u \cdot \nabla u\,
d\mu_{F} + \eps^2,
\end{eqnarray*}
and \eqref{BL} follows by letting $\eps$ tend to zero.

Let us go
back to interpolation families. As we said, $1$-sub-interpolation  corresponds to a
function $F$ that is convex on $I\times\R^{n}$. More generally, we
have the following characterization, proved by Semmes:

\begin{prop}
For a smooth function $F:I\times\R^{n}\to \R$, the following are
equivalent:
\begin{itemize}
\item $F$ is a sub-family of $p$-interpolation.
\item With the notation~\eqref{def:H}, we have, $\forall (t,x)\in I\times\R^n$,  $\D H_p F (t,x) \ge 0$.
\item For all $x_0, y_0\in \R^n$, the function
$$(s,t) \longto F \left(t,x_0 + (t+\sqrt{p-1} \, s) y_0 \right)$$
is subharmonic on the subset of $\R^2$ where it is defined.
\end{itemize} \label{prop_1132}
\end{prop}
Note that the third condition in Proposition \ref{prop_1132} needs
only a minimal level of smoothness. We may thus speak of a
sub-family $F$ of $p$-interpolation even when $F$ is not very
smooth.

We turn now to duality, which was part of the  motivation of Semmes.
We shall denote by $\L$ the Legendre transform in space, i.e. on
$\R^n$. In particular, for $F:I\times \R^n$, we shall write
$$\L F(t,x) = \L (F_t)(x) = \sup_{y\in \R^n} \big\{ x\cdot y - F(t,y) \big\}.$$
It is classical that if $F$ is the family of $1$-interpolation given by~\eqref{infconv}, then $\L F$ is a family of $\infty$-interpolation, meaning that $\L F$ is affine in $t$:
$$\L F_t (x) = (1-t)\L F_0(x) + t \L F_1(x).$$
So in this case,  when we move to the dual setting, Brunn-Minkowski or Pr\'ekopa's inequality is replaced by the trivial fact that $\alpha(t)=-\log\int e^{-\mathcal \L_t F(x)}\, dx$ is {\it concave} by  H\"older's inequality.

More general duality relations hold for $p$-interpolations.  Suppose $F(t, x)=F_t(x)$ is convex in $x$, and denote $G(t,y) = \L F_t(y)$. We have
the identity (proved below):
\begin{equation}
\partial^2_{tt} F + \partial^2_{tt} G =  (\Hess_x F)^{-1}
\nabla \partial_t F \cdot \nabla \partial_t F = (\Hess_y
G)^{-1}\nabla
\partial_t G \cdot \nabla \partial_t G, \label{ident}
\end{equation}
where $F$ and its derivatives are evaluated at $(t,x)$,
while $G$ and its derivatives are evaluated at $(t,y)=(t, \nabla F(x))$. From this
identity, we immediately conclude

\begin{prop}
If $F$ is a family of $p$-interpolation, then $\L F$ is a family of $p'$-interpolation, where $\frac1{p'}+\frac1p = 1$.
\end{prop}

We now present the details of the straightforward proof of
(\ref{ident}).
 From the definition,
\begin{eqnarray}
 G(t, \nabla F(t,x)) &=& \langle x, \nabla F(t,x) \rangle - F(t,x),
\label{eq_1043}
\\
 \nabla G_t (\nabla F_t(x)) = (\nabla G)(t, \nabla F(x))& =& x
\label{eq_1039}
\\
\Hess_{y} G(t, \nabla F(x,t)) &=& (\Hess_x F(t, x))^{-1}.
\label{eq_1040}
\end{eqnarray}
where the gradients and the hessians refer only to the space variables  $x,y$. By differentiating (\ref{eq_1039}) with respect to $t$, we see
that
\begin{equation}
\nabla \partial_t{G} = -(\Hess_y G) (\nabla \partial_t F)
\label{eq_1037}
\end{equation}
where $G$ and its derivatives are evaluated at $(t,y)=(t, \nabla
F(x))$, while $F$ and its derivatives are evaluated at $(t,x)$. From
(\ref{eq_1040}) and (\ref{eq_1037}),
\begin{equation}\label{dualformula0}
-\nabla \partial_t G \cdot \nabla\partial_t F= (\Hess_x F)^{-1} \nabla \partial_t F \cdot \nabla \partial_t F =  (\Hess_y G)^{-1} \nabla \partial_t G \cdot \nabla \partial_t G .
\end{equation}
 Differentiating
(\ref{eq_1043}) with respect to $t$ and using (\ref{eq_1039}) we get that
$
\partial_t G (t, \nabla F(x))= -\partial_t F(t,x)$.
If we differentiate this last equality one more time with respect to $t$, we find
$$
\partial^2_{tt}G + \nabla \partial_t G \cdot \nabla \partial_t F = -\partial^{2}_{tt}
F,
$$
which combined with~\eqref{dualformula0} yields the desired formula
(\ref{ident}).

\medskip
 As a consequence of Proposition \ref{prop_1132}, we see that $2$-interpolation families
satisfy an interpolation duality theorem. Let $f$ be a convex
function on $\R^n$, and suppose that $F_t(x) = F(t,x)$ is the
 $2$-interpolation
 family $F$ with $F_0 = f$ and $F_1 = \L f$. Then, $$ F(t,x) = \L F(1-t,x) $$
  provided we have unicity for the $2$-interpolation problem,
and therefore we have
$$F \left(\frac12, x \right) = \frac{|x|^2}{2}. $$ If we take $f(x)=\|x\|_K^2/2$, then $\L
f(x)=\|x\|_{K^\circ}^2/2$. Thus, if we could prove that  for a
$2$-interpolation family $F$, the associated function $\alpha$
from~\eqref{defalpha} is convex, as it is for $1$-interpolations,
then we would  recover Santal\'o's inequality. This would be the
case  if we had a Brascamp-Lieb inequality with a factor $\ffrac12$
on the right-hand side of~\eqref{BL} for every convex function
$F:\R^n\to \R$. However, this is  of course false in general. Recall
that even for the Santal\'o inequaliy, some ``center'' must be fixed
or some symmetry must be assumed. Therefore, a more reasonable
question to ask, is whether $\alpha$ is convex when the initial data
$f$ is even. This guarantees that $F_t$ is even for all $t\in
[0,1]$. However, it is again false in general that the Brascamp-Lieb
inequality holds with factor $\ffrac12$ in the right-hand side
of~\eqref{BL} when $F$ and $u$ are even, as  can be shown by taking
a perturbation of the Gaussian measure. This suggests that the
answer to the question could be negative in general. A reasonable
conjecture, perhaps, is:

\begin{conj}
Assume $F_0$ and $F_1$ are even, convex and $2$-homogeneous (i.e.
$F_i(x) = \lambda_i \|x\|_{K_i}^2$ for some centrally-symmetric
convex bodies $K_i\subset \R^n$), properties that propagate along
the interpolation. Then, the function $\alpha$ associated with the
$2$-interpolation family is convex.
\end{conj}

Here is a much more modest result:

\begin{fact}
Assume that $f$ is convex and even, and let $F$ be a
$2$-interpolation family with $F_0=f$ and $F_1=\L f$, with the
 associated function $\alpha$ as in~\eqref{defalpha}. Then, one has
$$\alpha''\left(1/2 \right) \ge 0.$$
\end{fact}

\begin{proof}
Since $F(\frac12,x)=|x|^2/2$, the probability measure $\mu_{F_{1/2}}$ is exactly the Gaussian measure on $\R^n$, which we denote by $\gamma$.
Note also that $\Hess_x F_{1/2}= \id_{\R^n}$. Therefore, if we denote $u=\partial_t F(\frac12, \cdot)$, we need to check that
$$\textrm{Var}_{\gamma}(u) \le  \frac12 \int_{\R^n}  |\nabla u|^2 \, d\gamma.$$
The function $v:=u- \int u\, d\gamma$ is by construction orthogonal
 to constant functions in $L^2(\gamma) $. But since $u$ is even
(because $F_t$ is even for all $t$, and so is $\partial_t F$), this
function $v$  is also orthogonal to linear functions. Recall that
the Hermite (or Ornstein-Uhlenbeck) operator  $L=\Delta - x\cdot
\nabla $ has non-positive integers as eigenvalues, and that the
eigenspaces (generated by Hermite polynomials) associated with the
eigenvalues $0$ and $-1$ are formed by the constant and linear
functions. Therefore, $v$ belongs to the subspace where $-L\ge 2\,
\id$ and so
$$\textrm{Var}_{\gamma}(u) = \int |v|^2 \, d\gamma \le - \frac12 \int v Lv \, d\gamma = \frac12 \int |\nabla u |^2 \, d\gamma.$$
 \end{proof}

We conclude this section by mentioning that we have analogous
formulas in the case where we work with some fixed measure $\nu$ on
$\R^n$, in place of the Lebesgue measure. Then,  for a function
$F:\R^n\to \R$ such that $\int e^{-F}\, d\nu<+\infty$, we denote by
$\mu_{\nu,F}$ the probability measure  on $\R^n$ given  by
$$d\mu_{\nu,F}(x):= \frac{\D e^{-F(x)}}{\int e^{-F}\, d\nu} \, d\nu(x).$$
For a function of $n+1$ variables  $F:I\times \R^n \to \R$,  we
denote as before   $F_t:=F(t,\cdot):\R^n\to \R$ and then $\mu_{\nu,
F_t}$ is the corresponding probability measure on $\R^n$.  We are
then interested in the convexity of the function
$$\alpha_{\nu}(t) := -\log \int_{\R^n} e^{-F(t,x)}\, d\nu(x) = - \log\int_{\R^n} e^{-F_t}\, d\nu.$$
The computation is identical:
$$\alpha_\nu''(t)=  \int_{\R^n} \partial^2_{tt} F \, d\mu_{\nu, F_t} - {\rm Var}_{\mu_{\nu, F_t}} \big( \partial_t F\big) .$$
Here is an illustration. Let $\nu$ be a symmetric log-concave
measure on $\R^n$: $d\nu(x)=e^{-W(w)}\, dx$ with $W$ being convex
and even on $\R^n$, and consider the family
$$F(t,x) =e^{t}\, |x|^2/2.$$
This is a typical example of a $2$-interpolation family. Then, the
fact that the corresponding $\alpha_\nu$ is convex is equivalent to
the $B$-conjecture proved in~\cite{CFM}. The argument there begins
with the computation above.  It turns out that for this particular
family $F$, the required Brascamp-Lieb inequality reduces to a
Poincar\'e inequality for the measure $\mu_{\nu,F_t}$, which holds
precisely with a constant $1/2$ when restricted to even functions.

Let us also mention in this direction that the Santal\'o inequality
in its functional form~\eqref{functsantalo} also holds if  the
Lebesgue measure is, in the three integrals, replaced by an even
log-concave measure of $\R^n$, as noted in Klartag~\cite{Klartag}.
Several examples of this type suggest that  the Lebesgue measure can
often be replaced by a more general log-concave measure.

\section{The Busemann Inequality}

We conclude this survey with a proof of the Busemann inequality via
$L^2$ inequalities. The Busemann inequality \cite{Bus} is concerned
with non-parallel hyperplane sections of a convex body $K \subset
\RR^n$. In the particular case where $K$ is centrally-symmetric, the
Busemann inequality states that
$$ g(x) = \frac{|x|}{|K \cap x^{\perp}|} \quad \quad \quad \quad (x
\in \RR^n) $$ is a norm on $\RR^n$. Here $|K \cap x^{\perp}|$ is the
$(n-1)$-dimensional volume of the hyperplane section $K \cap
x^{\perp} = \{ y \in K ; y \cdot x = 0 \}$, and $g(0) = 0$ as
interpreted by continuity.  The convexity of the function $g$ is a
non-trivial fact.  Using the Brunn-Minkowski inequality, the
convexity of $g$ reduces to a statement about log-concave functions
in the plane, as observed by Busemann. Indeed, the convexity of $g$
has to be checked along affine lines, and therefore on
$2$-dimensional vector subspaces. Specifically, let $E \subset
\RR^n$ be a two-dimensional plane, which we conveniently identify
with $\RR^2$. For $y \in \RR^2 = E$ set
$$ e^{-w(y)} = |K \cap (y + E^{\perp})|, $$
the $(n-2)$-dimensional volume of the the section of $K$. Then $w:
\RR^2 \rightarrow \RR \cup \{ +\infty \} $ is a convex function,
according to the Brunn-Minkowski inequality. For $p > 0$ and $t \in
\RR$ define
\begin{equation} \alpha_p(t) = \int_0^{\infty} e^{-w(t s,s)} s^{p-1} ds.
\label{eq_622} \end{equation} Note that when $K$ is
centrally-symmetric, $2 \sqrt{1 + t^2} \alpha_1(t) = |K \cap
(1,-t)^{\perp}|$. We therefore see that Busemann's inequality
amounts to the convexity of the function $1/\alpha_1(t)$ on $\RR$.
Next we will prove  the following  more general statement, which is
due to Ball \cite{Ball_studia} when $p \geq 1$:
\begin{theo} Let $X$ be an $n$-dimensional real linear space and
let $w: X \rightarrow \RR$ be a convex function with $\int
 e^{-w} < \infty$. For $p > 0$ and $0 \neq x \in X$ denote
$$ h(x) = \left( \int_0^{\infty} e^{-w(s x)} s^{p-1} ds \right)^{-1/p} $$
with $h(0) =0$. Then $h$ is a convex function on $X$.
\label{busemann}
\end{theo}

Busemann's proof of the case $p=1$ of Theorem \ref{busemann}, and
the generalization to $p \geq 1$ by Ball, rely on transportation of
measure in one dimension. The proof we present below may be viewed
as an infinitesimal version of Busemann's transportation argument.
This is reminiscent of the proof given in Ball, Barthe and Naor
\cite{Ball:2003p54} of the Pr\'ekopa inequality, which may be viewed
as an infinitesimal version of the transportation proof of the
latter inequality.

\medskip {\it Proof of Theorem \ref{busemann}:} By a standard approximation
argument, we may assume that $w$ is smooth and $1/R \leq \Hess(w)
\leq R$ at all points of $\R^n$, for some large constant $R > 1$.
Therefore $h$ is a continuous function, smooth outside the origin,
and homogeneous of degree one. Since convexity of a function
involves three collinear points contained in a two-dimensional
subspace, we may assume that $n=2$. Thus, selecting a point $0 \neq
z \in X$ and a direction $\theta \in X$, our goal is to show that
$\partial^2_{\theta \theta} h(z) \geq 0$ (since $h$ is homogeneous
of degree one, it suffices to consider the case $z \neq 0$). If
$\theta$ is proportional to $z$, then the second derivative vanishes
as $h$ is homogeneous of degree one. We may therefore select
coordinates $(t,x) \in \R^2 = X$, and identify $z = (0,1)$ and
$\theta = (1,0)$. With this identification, in order to prove the
theorem we need to show that
$$ \left( \alpha_p^{-1/p} \right)^{\prime \prime}(0) \geq 0, $$
where $\alpha_p$ is defined in (\ref{eq_622}). Equivalently, we need
to prove that at the origin,
\begin{equation}
 \partial^2_{tt} \alpha_p \leq \left(1 + \frac{1}{p} \right) \left( \partial_t \alpha_p \right)^2 / \alpha_p.
 \label{eq_624}
 \end{equation}
 We denote by $\mu$
the probability measure on $[0, \infty)$ whose density is
proportional to the integrable function $\exp(-w(0, x)) x^{p-1}$.
Similarly to Fact \ref{formula} above with $F(t,x)=w(tx,x)$, the desired inequality
(\ref{eq_624}) is equivalent to
\begin{equation}   {\rm Var}_{\mu}(x \partial_t w) \leq \int_0^{\infty} x^2
(\partial^2_{tt} w) d\mu(x) + \frac{1}{p} \left( \int_0^{\infty} x
(\partial_t w) d \mu(x) \right)^2.  \label{eq_653} \end{equation} We
will use the convexity of $w(t,x)$ via the inequality
$\partial^2_{tt} w \geq \left( \partial^2_{tx} w \right)^2 /
\partial^2_{xx} w$, which expresses the fact that $w_t(x) = w(t,x)$
is a sub-family of $1$-interpolation. Denote $u(x) = x \partial_t
w(0, x)$ and compute that $x \partial^2_{tx} w = u^{\prime} - u(x) /
x$ for $x > 0$. Hence, in order to prove (\ref{eq_653}), it suffices
to show that
\begin{equation}   {\rm Var}_{\mu}(u) \leq \int_0^{\infty} \frac{1}{\partial^2_{xx} w} \left(
u^{\prime}(x) - \frac{u(x)}{x} \right)^2 d\mu(x) + \frac{1}{p}
\left( \int_0^{\infty} u d \mu(x) \right)^2. \label{eq_654}
\end{equation}
We will prove (\ref{eq_654}) for any smooth function $u \in
L^2(\mu)$ (it is clear that the function $x \partial_t w(0,x)$ grows
at most polynomially at infinity, and hence belongs to $L^2(\mu)$).
By approximation, it suffices to restrict our attention to smooth
functions such that $u - \int u d \mu$ is compactly-supported in
$[0, \infty)$.  Consider the Laplace-type operator
$$L \vphi =
\vphi^{\prime \prime} -  \Big( \partial_x w(0,x) - \frac{p-1}{x}
\Big) \vphi^{\prime}= \vphi^{\prime \prime} -  \partial_x\Big(  w(0,x) - (p-1)\log(x)
\Big) \vphi^{\prime}.$$
 Integrating the ordinary differential
equation, we find a smooth function $\vphi$, with $\vphi^{\prime}(0)
= 0$ and $\vphi^{\prime}$ compactly-supported in $[0, \infty)$, such
that $L \vphi = u - \int u d \mu$. As before, we have the
integration by parts $\int (L\varphi) u \, d\mu = - \int \varphi' u' \, d\mu$ and
$$ \int_0^{\infty} (L \vphi)^2 d \mu = -\int_0^{\infty} \vphi^{\prime}(x) u^{\prime}(x) d \mu =\int_0^{\infty} (\vphi^{\prime \prime}(x) )^2 d
\mu + \int_0^{\infty} \left( \partial^2_{xx} w + \frac{p-1}{x^2}
\right)(\vphi^{\prime}(x) )^2 d \mu. $$ Let us abbreviate $w^{\prime
\prime} = \partial^2_{xx} w(0,x), E = \int u d \mu$ and also
$\langle f \rangle = \int_0^{\infty} f(x) d \mu(x)$. Then, by using
the above identities and by completing three squares (marked by wavy
 underline),
\begin{align*}  {\rm Var}_{\mu}(u) & =  -2 \langle u' \varphi'\rangle - \langle (L\varphi)^2 \rangle \\
& = \uwave{\left \langle -2
\vphi^{\prime} \left( u^{\prime} - \frac{u}{x} \right) \right
\rangle} -\left \langle \frac{2 \vphi^{\prime} u}{x}  \right \rangle
- \left \langle (\vphi^{\prime \prime})^2 + \uwave{w^{\prime \prime}
(\vphi^{\prime})^2} + \frac{p-1}{x^2} (\vphi^{\prime})^2 \right
\rangle \\ & \leq \uwave{\left \langle \frac{1}{w^{\prime \prime}}
\left( u^{\prime} - \frac{u}{x} \right)^2 \right \rangle} -2 \left
\langle \frac{ \vphi^{\prime} (L \vphi + E)}{x}  \right \rangle -
\left \langle (\vphi^{\prime \prime})^2 + \frac{p-1}{x^2}
(\vphi^{\prime})^2\right \rangle \\  &= \left \langle
\frac{1}{w^{\prime \prime}} \left( u^{\prime} - \frac{u}{x}
\right)^2 \right \rangle +  \uwave{ \left \langle2 \vphi^{\prime
\prime} \vphi^{\prime} / x  \right \rangle }-
\left \langle \frac{2 \vphi^{\prime} E}{x} + \uwave{(\vphi^{\prime \prime})^2} +  (p+1) \frac{(\vphi^{\prime})^2}{x^2} \right \rangle \\
& \leq \left \langle \frac{1}{w^{\prime \prime}} \left( u^{\prime} -
\frac{u}{x} \right)^2 \right \rangle - \left \langle \frac{2
\vphi^{\prime} E}{x} + \uwave{p\frac{(\vphi^{\prime})^2}{x^2}
}\right \rangle \leq \left \langle \frac{1}{w^{\prime \prime}}
\left( u^{\prime} - \frac{u}{x} \right)^2 \right \rangle +
\frac{E^2}{p},
\end{align*}
and (\ref{eq_654}) is proven.  \hfill $\square$


\end{document}